\documentclass[]{article}
\usepackage{dsfont}
\usepackage[cp1250]{inputenc}
\usepackage[T1]{fontenc}
\usepackage{theorem}
\usepackage{graphicx}
\usepackage[fleqn]{amsmath}
\usepackage{amsfonts}
\usepackage{amssymb}
\usepackage{bbm}
\usepackage[all]{xy}
\usepackage[hmargin=1.6in,vmargin=1.2in]{geometry}
\usepackage[affil-it]{authblk}

\newcounter{licznik}
\newcounter{temp}

\newtheorem{definition}{Definition} [licznik]
\newtheorem{theorem}[definition]{Theorem}

\newtheorem{lemma}[definition]{Lemma}
\newtheorem{remark}[definition]{Remark}

\newtheorem{corollary}[definition]{Corollary}

\newenvironment{proof}{\par\noindent {\bf Proof.}}
		{\begin{flushright} \vspace*{-1mm} \mbox{$\Box$} \end{flushright}}

\begin{document}

\title{From directed path to linear order \\
- the best choice problem for powers of directed path
\footnotetext{* The research was partially supported by MNiSW grant N N206 372939.}
}

\author{Andrzej Grzesik}
\affil{ Theoretical Computer Science Department\\
 Faculty of Mathematics and Computer Science\\
 Jagiellonian University\\
 ul. Prof. Stanis{\l}awa {\L}ojasiewicza 6, 30-348 Kraków, Poland\\
 {\rm andrzej.grzesik@uj.edu.pl}
}

\author{Micha{\l} Morayne, Ma{\l}gorzata Sulkowska$^*$}
\affil{{Institute of Mathematics and Computer Science\\
 Wroc{\l}aw University of Technology\\
 ul. Wybrze\.{z}e Wyspia\'{n}skiego 27, 50-370 Wroc{\l}aw, Poland\\
 {\rm \{michal.morayne, malgorzata.sulkowska\}@pwr.wroc.pl}}
}

\date{}

\maketitle
\rm
\small

\paragraph{Abstract.}
We examine the evolution of the best choice algorithm and the probability of its success from a directed path to the linear order of the same cardinality through $k$th powers of a directed path, $1 \leq k<n$. The vertices of a $k$th power of a directed path of a known length $n$ are exposed one by one to a selector in some random order. At any time the selector can see the graph induced by the vertices that have already come. The selector's aim is to choose online the maximal vertex (i.e. the vertex with no outgoing edges). It is shown that the probability of success $p_n$ for the optimal algorithm for the $k$th power of a directed path satisfies $p_n = \Theta(n^{-1/(k+1)})$. We also consider the case when the selector knows the distance in the underlying path between each two vertices that are joined by an edge in the induced graph. An optimal algorithm for this choice problem is presented. The exact probability of success when using this algorithm is given.

\vspace{0.2 cm}


\vspace{0.2 cm}


\normalsize

\stepcounter{licznik}
\paragraph{\thelicznik. ~Introduction.} The secretary problem is the most classical optimal stopping problem. One looks there for a strategy of choosing the best candidate from $n$ linearly ordered applicants for a job as a secretary. The selector knows the total number of candidates and examines them one by one in some random, unknown a priori, order. At time $t$ (when the $t$th applicant is being interviewed) all the relative ranks of the candidates examined so far are revealed, however nothing is known about the future candidates. Selector's aim is to hire the presently examined candidate maximizing the probability that this one is the best from the whole pool. This problem has a full solution which is of a threshold type. It tells the selector to wait until a certain moment (asymptotically $n/e$) and at this moment or later choose the first applicant which is the best up to now. The probability of success is asymptotically $1/e$. For the first time a full solution of this problem was written down by Lindley in \cite{DVL}.

Many different variants of this beautiful problem were later considered. For the historical overview of the classical secretary problem consult Ferguson's survey \cite{TF}. Partially ordered versions (where a linear order of candidates is replaced by a partial one and by success we understand choosing one of its maximal elements) were first considered by Stadje (\cite{WS}). Threshold strategies for poset version of the best choice problem were considered later by a group of Russian mathematicians in a series of papers. An account of this research is given by Gnedin in \cite{AVG}. Optimal strategies for regular or simple posets were found by Garrod, Kubicki and Morayne (\cite{GKM}) and Morayne (\cite{MM}), Ka{\'z}mierczak (\cite{WK1}, \cite{WK2}) and Tkocz(\cite{JT}) and Ka{\'z}mierczak and Tkocz (\cite{WKJT}). Preater (\cite{JP}) considered a restricted information case when the selector knows in advance only the total number of candidates and has no other information about the underlying poset. Surprisingly, he showed that even then it is possible to achieve success with probability bounded away from zero. Improvements of Preater's bound were obtained by Georgiou, Kuchta, Morayne and Niemiec (\cite{GKMN}), Kozik (\cite{JK}), Freij and W\"{a}stlund (\cite{FW}). Problems with still partial but reacher information were considered by Garrod and Morris (\cite{GM}), Kumar, Lattanzi, Vassilvitskii and Vattani (\cite{KLVV}).

Orders are very rich directed graphs where each pair of comparable elements is a directed edge (the direction is from a smaller to a bigger one). Thus in a linear order each pair is connected (Fig.\ref{fig_kthpower}c represents the linear order of length $4$). The structure of a directed path is much poorer, only consecutive elements are joined by a directed edge (Fig.\ref{fig_kthpower}a). In \cite{GKMM} Kubicki and Morayne found an optimal algorithm and its probability of success for choosing a last vertex from a directed path (here a selector can see at a given moment a graph induced by the vertices that have already arrived). The optimal stopping time for choosing one of two last vertices from a directed path was found by Przykucki and Sulkowska in \cite{MPMS}. The analogue of Preater's problem for graphs was investigated by Goddard, Kubicka and Kubicki (\cite{GKK}) and Sulkowska (\cite{MS}). Some further generalization to random graphs was considered by Przykucki in \cite{MP}.

One can think that on one end we have a directed path and on the other end a linear order which refers to the $(n-1)$st, i.e., full, power of a directed path (Fig.\ref{fig_kthpower}).
\begin{figure}[!ht]
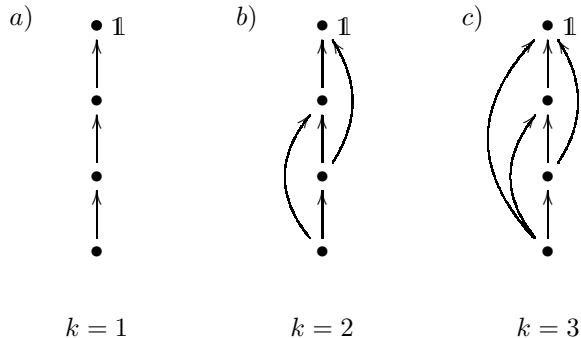

	\centerline{
		\xygraph{
		!{<0cm,0cm>;<1cm,0cm>:<0cm,1cm>::}
		!{(-1,4.1)}*+{a)}
		!{(0,1)}*+{\bullet}="11"
		!{(0,2)}*+{\bullet}="12"
		!{(0,3)}*+{\bullet}="13"
		!{(0,4)}*+{\bullet}="14"
		!{(0,4)}*+{\bullet}="14"
		!{(0.3,4)}*+{\mathds{1}}
		!{(2,4.1)}*+{b)}
		!{(3,1)}*+{\bullet}="21"
		!{(3,2)}*+{\bullet}="22"
		!{(3,3)}*+{\bullet}="23"
		!{(3,4)}*+{\bullet}="24"
		!{(3.3,4)}*+{\mathds{1}}
		!{(5,4.1)}*+{c)}
		!{(6,1)}*+{\bullet}="31"
		!{(6,2)}*+{\bullet}="32"
		!{(6,3)}*+{\bullet}="33"
		!{(6,4)}*+{\bullet}="34"
		!{(6.3,4)}*+{\mathds{1}}	
		!{(0,0)}*+{k=1}
		!{(3,0)}*+{k=2}
		!{(6,0)}*+{k=3}
		"11":"12"
		"12":"13"
		"13":"14"
		"21":"22"
		"22":"23"
		"23":"24"
		"21":@/^0.5cm/"23"
		"22":@/_0.4cm/"24"
		"31":"32"
		"32":"33"
		"33":"34"
		"31":@/^0.5cm/"33"
		"31":@/^0.8cm/"34"
		"32":@/_0.4cm/"34"
		}
  }
  \caption{The $k$th powers of a directed path of length $4$.}
  \label{fig_kthpower}
\end{figure}
As mentioned above, the best choice problems for those two cases have been solved (\cite{DVL}, \cite{GKMM}). It is a natural question what happens in between. Here we address this question by considering, we believe, the most natural evolution from a directed path to a linear order. Namely, we consider the $k$th powers, $k=1,2,\ldots,n-1$, of a directed path of $n$ elements. We show that the probability of success of the optimal algorithm for choosing a maximal element from the $k$th power of a directed path of length $n$ is of the order $n^{-1/(k+1)}$. We also find an optimal algorithm when the selector knows the distance in the underlying path between each two vertices that are joined by an edge in the induced graph.

\setcounter{temp}{\value{licznik}} The paper is organized as follows. In Section \stepcounter{temp}{\thetemp} we introduce basic definitions and notations. In Section \stepcounter{temp}{\thetemp} we find an optimal algorithm $\tau_n$ for the problem when the selector knows distances between connected vertices. In Section \stepcounter{temp}{\thetemp} we show that the probability of success of $\tau_n$, let us call it $p_n$, satisfies $p_n = \Theta(n^{-1/(k+1)})$ regardless of whether the selector knows distances. Section \stepcounter{temp}{\thetemp} discusses separately the case $k=n-1$ where our graph problem turns out to be the classical secretary problem with extra information.

\stepcounter{licznik}
\paragraph{\thelicznik. ~Definitions, notation and formal model.} A {\it directed graph} $G$ is a pair $G=(V,E)$, where $V$ is a set of {\it vertices} and $E$ is a set of {\it edges}, i.e., ordered pairs of elements from $V$ (which means that each edge has a direction). A {\it directed path} is a directed graph $P_n=(V_n, E_n)$ such that $V_n = \{v_1, v_2, \ldots, v_n\}$ and $E_n = \{(v_i, v_{i-1}): i \in \{2, 3, \ldots, n\} \}$. The {\it length} of $P_n$ is $n$. The {\it $k$th power} of a graph $G=(V,E)$ is the graph with the set of vertices $V$ and an edge between two vertices if and only if there is a path of length at most $k+1$ between them in $G$. We call $v \in V$ a {\it maximal element} or a {\it sink} if $v$ has no outgoing edges. For a directed graph $G$ the set of its maximal elements will be denoted by $Max(G)$ or $Max(V)$ if $E$ is known from the context. There is always only one sink in any power of a directed path and it will be denoted by $\mathds{1}$. For a graph $G=(V,E)$ its induced subgraph $G' = (W, E \cap W^2)$, $W \subseteq V$, is called a {\it connected component} if it is a maximal connected induced subgraph.

Let $\mathbb{N}$ denote the set of natural numbers, i.e., $\mathbb{N}=\{0, 1, 2, 3, \ldots\}$. Let us define a function $d_G:E \rightarrow \mathbb{N}$ by $d_G((v,w)) = l_G((v,w))-2$ where $l_G((v,w))$ is the length of the longest directed path in $G$ joining the vertices $v$ and $w$.

Let $G=(V,E)$ be a directed graph and let $S_n$ denote the family of all permutations of the set $V$. Let $\pi=(\pi_1, \pi_2, \ldots, \pi_n) \in S_n$. By $G_{(m)} = G_{(m)}(\pi) = (V_{(m)}, E_{(m)})$, $m \leq n$, we denote the subgraph of $G$ induced by $\{\pi_1, \ldots, \pi_m\}$, i.e.,
\[
\begin{split}
& V_{(m)} = \{\pi_1, \pi_2, \ldots, \pi_m\},\\
& E_{(m)} = \{(v_i, v_j) \colon \{v_i, v_j\} \subseteq \{\pi_1, \pi_2, \ldots, \pi_m\} \wedge (v_i, v_j) \in E \}.\\
\end{split}
\]
By $c(G_{(m)})$ we denote the number of connected components in $G_{(m)}$.

Let $(v_1, v_2, \ldots, v_m)$ be a sequence of distinct vertices of a directed graph $G = (V,E)$. Let $R \subseteq {\mathbb{N}}^2$. We write $(v_1, v_2, \ldots, v_m) \cong R$ if for all $i,j \leq m, i \neq j,$ $(v_i, v_j) \in E$ if and only if $(i,j) \in R$.

We will work with the probability space $(\Omega,\cal F,\mit P)$, where $\Omega = S_n$, ${\cal F} = {\cal{P}} (\Omega)$ and the probability measure $P:{\cal F} \rightarrow [0,1]$ is defined by $P(\{\pi\}) = 1/n!$ for each $\pi \in S_n$. Let
\begin{equation*}
{\cal F}_t = \sigma\{\{\pi \in \Omega:(\pi_1, \pi_2, \ldots, \pi_t) \cong R\}: R \subseteq {\mathbb{N}}^2\}, \hspace{20pt} 1 \leq t\leq n,
\end{equation*}
be our \textit{filtration} (a sequence of $\sigma$-algebras such that $\cal F_{\mit 1} \subseteq \cal F_{\mit 2} \subseteq  \ldots \cal F_{\mit n} \subseteq \cal F$). We call a random variable $\tau : \Omega \rightarrow \{1, 2, \ldots, n \}$ a \textit{stopping time} with respect to the filtration $(\cal F_{\mit t})^{\mit n}_{\mit t=1}$ if $\tau^{-1}(\{t\}) \in \cal F_{\mit t}$ for each $t \leq n$ (which means that the decision to stop is based only on past and present events). Let $D$ be a subset of vertices of the graph $G$ (i.e. $D \subseteq V$). An {\it optimal stopping time} for choosing an element from $D$ is any stopping time $\tau^*$ for which
\begin{equation*}
\mathbb{P}[\pi_{\tau^*} \in D] = \max \limits_{\tau \in {\cal T}} \mathbb{P}[\pi_{\tau} \in D],
\end{equation*}
where $\cal T$ is the set of all stopping times and $[\pi_{\tau} \in D]$ denotes the set $\{\pi~\in~\Omega~\colon \pi_{\tau(\pi)} \in D\}$. Throughout this paper $G$ will always be a power of a directed path and $D = \{\mathds{1}\}$. In the next section we also assume that the selector knows the value $d_G$ of each edge that appears in the induced graph.

\stepcounter{licznik}
\paragraph{\thelicznik. ~Optimal stopping time.}
Let $P_{n}^{k} = (V_{n}^{k}, E_{n}^{k})$ be the $k$th power of the directed path $P_n$ ($1 \leq k < n$). The first, the second and the third powers of the directed path $P_4$ are in Fig.\ref{fig_kthpower}a,\ref{fig_kthpower}b,\ref{fig_kthpower}c respectively. (Whenever the context is clear we omit the indices $n$ and $k$ for clarity of notation and write $P$ instead of $P_{n}^{k}$.) In this section we find an optimal stopping time $\tau_{n}$ for choosing the sink from $P_{n}^{k}$.

Let $\pi \in S_n$ be a random permutation of vertices from $V_{n}^{k}$ and $P_{(t)}$ be the graph induced by $\{ \pi_1, \pi_2, \ldots, \pi_t \}$. Suppose that $H = (W,F)$ is a connected component in $P_{(t)}$ and that $w$ and $z$ are two extreme vertices of $H$. Since the value $d_P(e)$ is known for each $e \in F$, one can tell how many of the remaining vertices are going to be placed between $w$ and $z$ on $P_n^k$. Let $b_t$ be the number of those remaining ``inner'' vertices. (Compare Fig.\ref{fig_kP}.)

Let $\tau_n(\pi) =
\min\{t \leq n: n-t = k (c(P_{(t)}) - 1) + b_t, \hspace{3 pt} \pi_t \in Max\{\pi_1, \pi_2, \ldots \pi_t\}\}$, using the convention $\min\emptyset = n$.

Note that $\tau_n$ tells the selector not to stop as long as there is still a chance to win in the future. (For instance, we have $\tau_9=6$ in Fig.\ref{fig_kP}.) The condition $n-t = k (c(P_{(t)}) - 1) + b_t$ means that the probability that $\mathds{1}$ is still to come is equal to zero because among $n-t$ remaining vertices we need at least $k (c(P_{(t)}) - 1)$ vertices to connect the components that we have at the time $t$ and $b_t$ is exactly the number of vertices that will join already existing components falling somewhere between their vertices. Thus the strategy $\tau_n$ can be stated exactly as the analogue of the optimal strategy for a directed path from \cite{GKMM}.

\vspace{0.2cm}

{\it Stop when there is a positive conditional (given history) probability that the presently examined candidate is the sink and the probability that the sink can be among the future candidates is equal to zero.}

\begin{theorem}
\label{th_optimal}
Let $\pi$ be a random permutation of vertices of $P_n^k$. For $P_n^k$, $1 \leq k < n$, the stopping time $\tau_n$ is optimal, i.e., 
$$
\mathbb{P}[\pi_{\tau_n} = \mathds{1}] = \max_{\tau \in \cal{T}} \mathbb{P}[\pi_{\tau} = \mathds{1}],
$$
where $\cal{T}$ is the set of all stopping times.
\end{theorem}
\begin{proof}\textbf{1}
First, let us observe that it is reasonable to stop at time $m$ only if $\pi_m \in Max(P_{(m)})$. Of course, we should stop if $\mathbb{P}[\mathds{1} \in \{\pi_{m+1}, \ldots, \pi_n\}|\pi_m \in Max(P_{(m)})] = 0$. From \cite{GKMM} we know that if we play on a directed path and $\mathbb{P}[\mathds{1} \in \{\pi_{m+1}, \pi_{m+2}, \ldots, \pi_n\}|\pi_m \in Max(P_{(m)})] > 0$ then it always pays off to play further, for instance simple waiting for the next maximal element in the induced graph is profitable. Now we explain the correspondence between the game on a directed path and the game on its $k$th power. Since we assume that playing on the $k$th power one has the additional information $d_P$, one knows that at least $b_m$ of the remaining vertices are ``dummy''. They do not play any role in our game since we know that they are not going to appear as the maximal ones in the induced graph. We have $k (c(P_{(m)})-1)$ more ``dummy'' vertices that will appear immediately under the components seen a time $t=m$ (they are also not going to appear as the maximal ones in the induced graph). Note that it corresponds to the directed path case ($k=1$) at time $\tilde{m} = m + b_m + (k-1)c(P_{(m)}-1)$ when $\pi_{\tilde{m}}$ is maximal in the induced graph, the number of components of the induced graph is $c(P_{(m)})$ and we know about $c(P_{(m)})-1$ ``dummy'' vertices (supporting components at time $t= \tilde{m}$). Recall that probability that $\mathds{1}$ is still to come is positive thus as in the directed path case ($k=1$) we should play further(\cite{GKMM}) we should also play further in the $k$th power case since throughout the game we are going to obtain at least as much information as playing in the case $k=1$.
\end{proof}

\begin{figure}[!ht]
\centerline{
\xygraph{
!{<0cm,0cm>;<1cm,0cm>:<0cm,1cm>::}
!{(5.5,0.5)}*+{\bullet{\pi_1}}
!{(9.5,0.5)}*+{\bullet{\pi_1}}
!{(10.5,0.5)}*+{\bullet{\pi_2}}
!{(14,1)}*+{\bullet{\pi_1}}="1"
!{(14,0)}*+{\bullet{\pi_3}}="3"
!{(15,0.5)}*+{\bullet{\pi_2}}
"3":"1" ^{1}
!{(5.5,-0.8)}*+{t=1}
!{(5.5,-1.3)}*+{b_1=0}
!{(10,-0.8)}*+{t=2}
!{(10,-1.3)}*+{b_2=0}
!{(14.5,-0.8)}*+{t=3}
!{(14.5,-1.3)}*+{b_3=1}
}
}
\end{figure}
\begin{figure}[!ht]
\centerline{
\xygraph{
!{<0cm,0cm>;<1cm,0cm>:<0cm,1cm>::}
!{(5,1)}*+{\bullet{\pi_3}}="13"
!{(5,2)}*+{\bullet{\pi_1}}="11"
!{(6,1)}*+{\bullet{\pi_2}}="12"
!{(6,2)}*+{\bullet{\pi_4}}="14"
"13":"11" ^{1}
"12":"14" ^{1}
!{(9.5,0.5)}*+{\bullet{\pi_3}}="23"
!{(9.5,1.5)}*+{\bullet{\pi_5}}="25"
!{(9.5,2.5)}*+{\bullet{\pi_1}}="21"
!{(10.5,1)}*+{\bullet{\pi_2}}="22"
!{(10.5,2)}*+{\bullet{\pi_4}}="24"
"23":"25" ^{0}
"25":"21" ^{0}
"23":@/^0.6cm/"21" ^{1}
"22":"24" ^{1}
!{(13.7,0)}*+{\bullet{\pi_3}}="33"
!{(13.7,1)}*+{\bullet{\pi_5}}="35"
!{(13.7,2)}*+{\bullet{\pi_1}}="31"
!{(13.7,3)}*+{\bullet{\pi_6}}="36"
!{(15.3,1)}*+{\bullet{\pi_2}}="32"
!{(15.3,2)}*+{\bullet{\pi_4}}="34"
"33":"35" ^{0}
"35":"31" ^{0}
"31":"36" ^{0}
"33":@/^0.6cm/"31" ^{1}
"35":@/_0.6cm/"36" _{1}
"32":"34" ^{1}
!{(5.5,-0.8) }*+{t=4}
!{(5.5,-1.3) }*+{b_4=2}
!{(10,-0.8) }*+{t=5}
!{(10,-1.3) }*+{b_5=1}
!{(14.5,-0.8) }*+{t=6}
!{(14.5,-1.3) }*+{b_6=1}
}
}
\end{figure}
\begin{figure}[!ht]
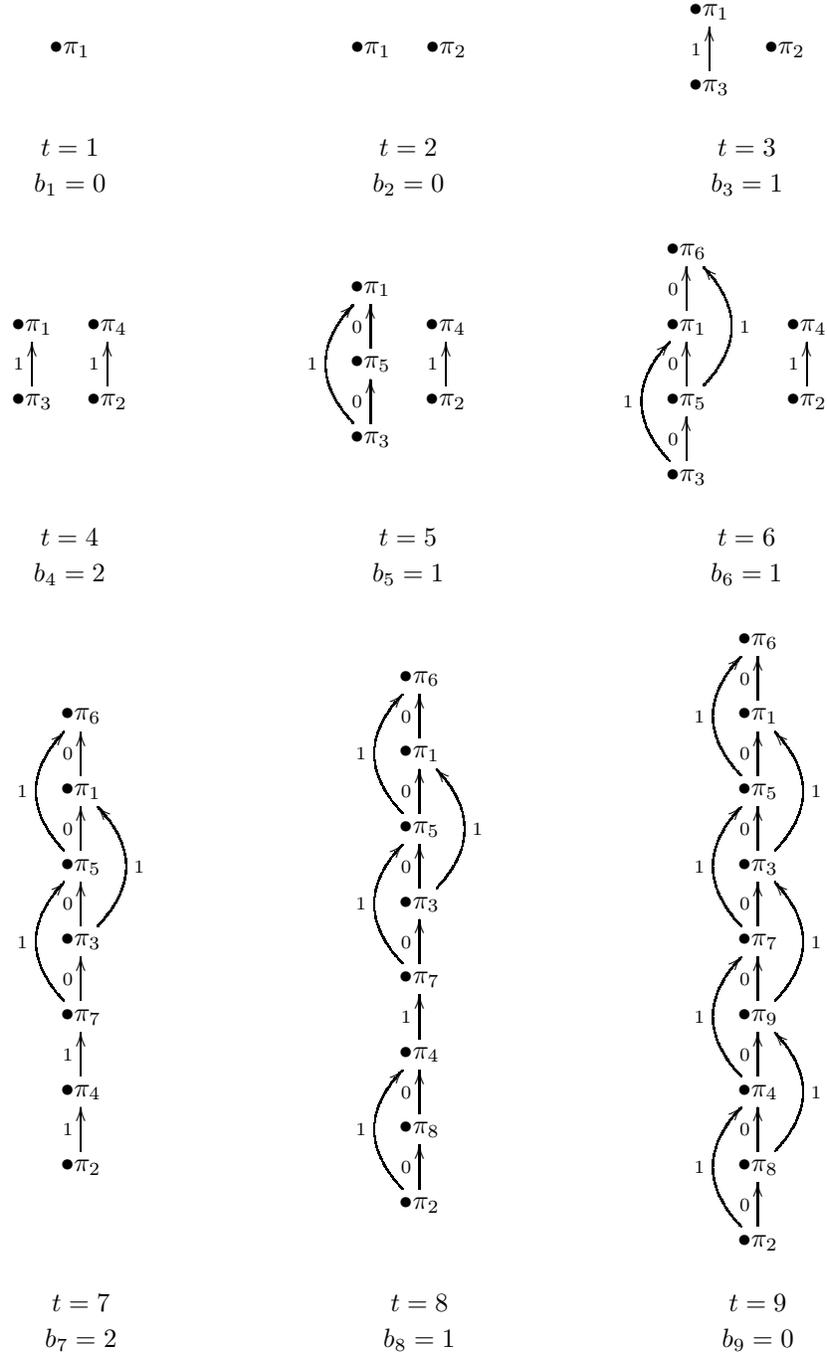

\centerline{
\xygraph{
!{<0cm,0cm>;<1cm,0cm>:<0cm,1cm>::}
!{(5.5,1)}*+{\bullet{\pi_2}}="12"
!{(5.5,2)}*+{\bullet{\pi_4}}="14"
!{(5.5,3)}*+{\bullet{\pi_7}}="17"
!{(5.5,4)}*+{\bullet{\pi_3}}="13"
!{(5.5,5)}*+{\bullet{\pi_5}}="15"
!{(5.5,6)}*+{\bullet{\pi_1}}="11"
!{(5.5,7)}*+{\bullet{\pi_6}}="16"
"12":"14" ^{1}
"14":"17" ^{1}
"17":"13" ^{0}
"13":"15" ^{0}
"15":"11" ^{0}
"11":"16" ^{0}
"17":@/^0.6cm/"15" ^{1}
"13":@/_0.6cm/"11" _{1}
"15":@/^0.6cm/"16" ^{1}
!{(10,0.5)}*+{\bullet{\pi_2}}="22"
!{(10,1.5)}*+{\bullet{\pi_8}}="28"
!{(10,2.5)}*+{\bullet{\pi_4}}="24"
!{(10,3.5)}*+{\bullet{\pi_7}}="27"
!{(10,4.5)}*+{\bullet{\pi_3}}="23"
!{(10,5.5)}*+{\bullet{\pi_5}}="25"
!{(10,6.5)}*+{\bullet{\pi_1}}="21"
!{(10,7.5)}*+{\bullet{\pi_6}}="26"
"22":"28" ^{0}
"28":"24" ^{0}
"24":"27" ^{1}
"27":"23" ^{0}
"23":"25" ^{0}
"25":"21" ^{0}
"21":"26" ^{0}
"22":@/^0.6cm/"24" ^{1}
"27":@/^0.6cm/"25" ^{1}
"23":@/_0.6cm/"21" _{1}
"25":@/^0.6cm/"26" ^{1}
!{(14.5,0)}*+{\bullet{\pi_2}}="32"
!{(14.5,1)}*+{\bullet{\pi_8}}="38"
!{(14.5,2)}*+{\bullet{\pi_4}}="34"
!{(14.5,3)}*+{\bullet{\pi_9}}="39"
!{(14.5,4)}*+{\bullet{\pi_7}}="37"
!{(14.5,5)}*+{\bullet{\pi_3}}="33"
!{(14.5,6)}*+{\bullet{\pi_5}}="35"
!{(14.5,7)}*+{\bullet{\pi_1}}="31"
!{(14.5,8)}*+{\bullet{\pi_6}}="36"
"32":"38" ^{0}
"38":"34" ^{0}
"34":"39" ^{0}
"39":"37" ^{0}
"37":"33" ^{0}
"33":"35" ^{0}
"35":"31" ^{0}
"31":"36" ^{0}
"32":@/^0.6cm/"34" ^{1}
"38":@/_0.6cm/"39" _{1}
"34":@/^0.6cm/"37" ^{1}
"39":@/_0.6cm/"33" _{1}
"37":@/^0.6cm/"35" ^{1}
"33":@/_0.6cm/"31" _{1}
"35":@/^0.6cm/"36" ^{1}
!{(5.5,-0.8)}*+{t=7}
!{(5.5,-1.3)}*+{b_7=2}
!{(10,-0.8)}*+{t=8}
!{(10,-1.3)}*+{b_8=1}
!{(14.5,-0.8)}*+{t=9}
!{(14.5,-1.3)}*+{b_9=0}
}
}
\caption{Induced graphs of $P_9^2$ at time $t$ for $\pi = (v_2, v_9, v_4, v_7, v_3, v_1, v_5, v_8, v_6)$.}
\label{fig_kP}
\end{figure}

As sometimes the intuitive type argument may contain a hidden bug, to be on the safe side, we also present below a fully formal proof of the optimality of $\tau_n$.

\vspace{0.2cm}

\begin{proof}\textbf{2}
This proof is analogous to the one that shows the optimality of $\tau_n$ for $k=1$ presented in \cite{GKMM}. At first, let us observe that it is reasonable to stop only if the currently examined vertex is maximal in the induced graph. Now, aiming for a contradiction, let us assume that there exists a stopping time $\tau$ such that $\mathbb{P}[\pi_{\tau}=\mathds{1}] > \mathbb{P}[\pi_{\tau_n}=\mathds{1}]$ which is optimal and that 
there is no optimal stopping time $\tilde{\tau} \geq \tau$ and $\tilde{\tau} \neq \tau$. By our observation we may also assume that $\tau(\pi)=t$ if and only if $\pi_t \in Max(P_{(t)})$ or $t=n$.

Whenever $\pi_n = \mathds{1}$ we have $\tau_n(\pi) = \mathds{1}$ thus
\[
\mathbb{P}[\pi_{\tau} = \mathds{1}|\tau=n] \leq \mathbb{P}[\pi_{\tau_n}=\mathds{1}|\tau = n].
\]
Hence let us now consider the event $[\tau < n]$.

We have $\tau(\pi) = m < n$ and $\pi_m \in Max(P_{(m)})$. Let $a_m = k(c(P_{(m)})-1)$. Let us calculate the probability that $\tau$ wins counting simply all the possible settings of the remaining vertices. We need at least $a_m$ out of the remaining vertices to connect the components of $P_{(m)}$ (which refers to the term ${n-m \choose a_m} a_m!$ in (\ref{eq_pi_m})). Moreover, we need $b_m$ more vertices out of the remaining ones that will fall between the extreme vertices of the components in $P_{(m)}$ (which refers to the term ${n-m-a_m \choose b_m} b_m!$). Finally, all the $n-m-a_m-b_m$ remaining vertices may be arbitrarily permuted together with $c(P_{(m)})$ components (which refers to the term $(n-m-a_m-b_m+c(P_{(m)}))!$). If we wish to have the component containing $\pi_m$ at the top of the whole graph, then we can arbitrarily permute the $n-m-a_m-b_m$ remaining vertices with $c(P_{(m)})-1$ components (which refers to the term $(n-m-a_m-b_m+c(P_{(m)})-1)!$).
Hence we get
\begin{equation}
\label{eq_pi_m}
\begin{split}
\mathbb{P}[\pi_m = & \mathds{1}|\pi_m \in Max(P_{(m)})] = \\
& \frac{{n-m \choose a_m} a_m! {n-m-a_m \choose b_m} b_m! (n-m-a_m-b_m+c(P_{(m)})-1)!} {{n-m \choose a_m} a_m! {n-m-a_m \choose b_m} b_m! (n-m-a_m-b_m+c(P_{(m)}))!} =\\
& \frac{1}{n-m-a_m-b_m+c(P_{(m)})}.
\end{split}
\end{equation}
Since all the components of $P_{(m)}$ have the same chance to be placed at the top of the whole underlying graph, we obtain
\[
\mathbb{P}[\mathds{1} \in P_{(m)}|\pi_m \in Max(P_{(m)})] = \frac{c(P_{(m)})}{n-m-a_m-b_m+c(P_{(m)})}
\]
which implies
\[
\mathbb{P}[\mathds{1} \notin P_{(m)}|\pi_m \in Max(P_{(m)})] = \frac{n-m-a_m-b_m}{n-m-a_m-b_m+c(P_{(m)})}.
\]
Let us consider the following stopping time
\[
\overline{\tau}(\pi) = \left\{  \begin{array}{ll}
\min\{t>m: \pi_t \in Max(P_{(t)}) \} & \textrm{if } \tau(\pi) = m < n,\\
n & \textrm{in the remaining cases,}
\end{array}\right.
\]
using the convention $\min\emptyset = n$. Because $\tau \neq \tau_n$ there exists $m$ such that $\{t>m: \pi_t \in Max(P_{(t)}) \} \neq \emptyset$. We will show that 
\[
\mathbb{P}[\pi_{\overline{\tau}(\pi)}=\mathds{1}|\pi_m \in Max(P_{(m)})] \geq \mathbb{P}[\pi_m = \mathds{1}|\pi_m \in Max(P_{(m)})].
\]
Note that among $n-m$ vertices that are still to come there are at most $n-m-a_m-b_m$ which may arrive as the maximal ones in the induced graph. Therefore if $\mathds{1}$ is among the remaining vertices then with probability at least $1/(n-m-a_m-b_m)$ it will appear as the first maximal vertex in the induced graph after time $m$ (note that whenever $\mathds{1}$ is among the remaining vertices, $n-m-a_m-b_m>0$). Therefore
\[
\begin{split}
\mathbb{P}[\pi_{\overline{\tau}(\pi)}=& \mathds{1}|\pi_m \in Max(P_{(m)})] = \\ & \mathbb{P}[\overline{\tau}(\pi)=\mathds{1}|\mathds{1} \notin P_{(m)}, \pi_m \in Max(P_{(m)})] \mathbb{P}[\mathds{1} \notin P_{(m)}|\pi_m \in Max(P_{(m)})] \geq \\
& \frac{1}{(n-m-a_m-b_m)} \frac{(n-m-a_m-b_m)}{(n-m-a_m-b_m+c(P_{(m)}))} = \\
& \mathbb{P}[\pi_m = \mathds{1}|\pi_m \in Max(P_{(m)})].
\end{split}
\]
Thus we have found the stopping time $\overline{\tau}$ which is at least equally effective as $\tau$ and stops later than $\tau$ which contradicts the assumption that there is no optimal stopping time $\tilde{\tau} \geq \tau$, $\tilde{\tau} \neq \tau$. This proves the optimality of $\tau_n$.
\end{proof}

\stepcounter{licznik}
\paragraph{\thelicznik. ~The probability of success.}
In this section we show that the probability of success of the optimal algorithm for choosing the sink from $P_n^k$ is of the order $n^{-1/(k+1)}$, $1 \leq k < n$. (Again we write for short $P$ instead of $P_n^k$.) We show this result also for the case when the selector knows the values of $d_P$ in the induced graph. We give the exact formula for the probability of success in the latter case.

\begin{theorem}
\label{th_opt_prob}
Let $\pi$ be a random permutation of vertices of $P_n^k$. Then
\begin{equation}
\label{eq_opt_probabil}
\begin{split}
\mathbb{P}[\pi_{\tau_n} = \mathds{1}] = & \sum_{m = \lceil \frac{n+k}{k+1} \rceil}^{n} \frac{1}{m{n \choose m}}\sum_{h=0}^{\lfloor (n-m)/k \rfloor} \sum\limits_{\begin{subarray}{c} a_1,a_2,\ldots,a_{k-1}: \\ a_1+2 a_2 + \ldots + (k-1)a_{k-1}=\\n-m-kh\end{subarray}}{m-1 \choose h+a_1+a_2+\ldots+a_{k-1}} \cdot \\
& {h+a_1+\ldots+a_{k-1} \choose {h,a_1,a_2,\ldots,a_{k-1}}}.
\end{split}
\end{equation}
\end{theorem}
\begin{proof}
Let
\[
\begin{split}
& B_m = [\pi_m \in Max\{\pi_1, \pi_2, \ldots, \pi_m\}],\\
& C_m = [n-m = k (c(P_{(m)}) - 1) + b_m],\\
& A_m = B_m \cap C_m.
\end{split}
\]
Since $C_m = \emptyset$ for $m<(n+k)/(k+1)$, we have
\[
\mathbb{P}[\pi_{\tau_n} = \mathds{1}] = \sum_{m = \lceil \frac{n+k}{k+1} \rceil}^{n} \mathbb{P}[\pi_{\tau_n}=\mathds{1}|A_m]\mathbb{P}[A_m] = \sum_{m = \lceil \frac{n+k}{k+1} \rceil}^{n} \mathbb{P}[\pi_{\tau_n}=\mathds{1}|A_m]\mathbb{P}[B_m|C_m] \mathbb{P}[C_m].
\]
Note that $C_m$ means that at the time $m$ all the remaining vertices are going to fall between the vertices of $P_{(m)}$ (none of the remaining vertices can be one of the two extreme vertices of $P_n$). Note that, since we deal with the $k$th power of a directed path, not more than $k$ vertices of those that are still to come can be finally placed in $P_n$ next to each other. Let
\[
V_{m,h} = \sum\limits_{\begin{subarray}{c} a_1,a_2,\ldots,a_{k-1}:\\ a_1+ 2 a_2\ldots + (k-1)a_{k-1}=\\n-m-kh\end{subarray}}{m-1 \choose h+a_1+\ldots+a_{k-1}} {h+a_1+\ldots+a_{k-1} \choose {h,a_1,a_2,\ldots,a_{k-1}}}.
\]
We have (it is explained below)
\begin{equation}
\label{eq0}
\begin{split}
\mathbb{P}[C_m] = & \frac{1}{{n \choose m}} \sum_{h=0}^{\lfloor (n-m)/k \rfloor} V_{m,h}.
\end{split}
\end{equation}
In this formula the $h$th term corresponds to $P_{(m)}$ having $h+1$ components. The term ${m-1 \choose h+a_1+\ldots+a_{k-1}}$ refers to the way we choose spaces between $m$ vertices of $P_{(m)}$ for the elements that are still to come ($h$ spaces that will consist of $k$ vertices, $a_i$ spaces that will consist of $i$ elements for $i=1,2,\ldots,k-1$).
The term ${h+a_1+\ldots+a_{k-1} \choose h}$ refers to choosing that spaces that consist of $k$ vertices and separate components. The expression ${a_1+\ldots+a_{k-1} \choose a_1} {a_2+\ldots+a_{k-1} \choose a_2}\ldots {a_{k-1} \choose a_{k-1}}$ refers to choosing $a_1$ spaces that consist of $1$ vertex, $a_2$ spaces that consist of $2$ spaces and so on. Since we need $kh$ of the remaining vertices to form $h$ spaces, we are left with $n-m-kh$ vertices that we use to form spaces of cardinality less than $k$ (which explains why we sum over $a_1, a_2, \ldots, a_{k-1}$ satisfying $a_1+a_2+\ldots+a_{k-1}=n-m-kh$). From $n-m$ remaining vertices we may form at most $\lfloor (n-m)/k \rfloor$ groups of $k$ vertices, which explains the upper limit of summation. Let $W_m = \sum_{h=0}^{\lfloor (n-m)/k \rfloor} V_{m,h}$.

We have
\[
\mathbb{P}[B_m| C_m] = \sum_{h=0}^{\lfloor (n-m)/k\rfloor} \mathbb{P}[B_m|[c(P_{(m)})=h+1] \cap C_m] \mathbb{P}[c(P_{(m)})=h+1|C_m].
\]
Obviously,
\begin{equation}
\label{eq1}
\mathbb{P}[B_m|[c(P_{(m)})=h+1] \cap C_m] = \mathbb{P}[B_m|c(P_{(m)})=h+1] = \frac{|Max(P_{(m)})|}{m} = \frac{h+1}{m}.
\end{equation}
We also have
\begin{equation}
\label{eq2}
\mathbb{P}[c(P_{(m)})=h+1|C_m] = \frac{V_{m,h}}{W_m}.
\end{equation}
Thus
\begin{equation}
\label{eq3}
\mathbb{P}[B_m|C_m] = \frac{1}{m W_m} \sum_{h=0}^{\lfloor (n-m)/k\rfloor}(h+1)V_{m,h}.
\end{equation}
Let $T_m = \sum_{h=0}^{\lfloor (n-m)/k\rfloor}(h+1)V_{m,h}$. We have
\[
\mathbb{P}[\pi_{\tau_n} = \mathds{1}|A_m] = \sum_{h=0}^{\lfloor (n-m)/k\rfloor} \mathbb{P}[\pi_{\tau_n} = \mathds{1}|[c(P_{(m)})=h+1] \cap A_m] \mathbb{P}[c(P_{(m)})=h+1 | A_m]
\]
and
\[
\mathbb{P}[\pi_{\tau_n} = \mathds{1}|[c(P_{(m)})=h+1] \cap A_m] = \frac{1}{h+1},
\]
and, by (\ref{eq0}), (\ref{eq1}), (\ref{eq2}) and (\ref{eq3}),
\[
\mathbb{P}[c(P_{(m)}) = h+1 | A_m] = \frac{(h+1)V_{m,h}}{T_m}.
\]
Hence
\[
\mathbb{P}[\pi_{\tau_n} = \mathds{1}|A_m] = \sum_{h=0}^{\lfloor (n-m)/k\rfloor} \frac{1}{h+1} \frac{(h+1)V_{m,h}}{T_m} = \frac{W_m}{T_m}.
\]
Thus
\[
\mathbb{P}[\pi_{\tau_n} = \mathds{1}|A_m]\mathbb{P}[B_m|C_m]\mathbb{P}[C_m] = \frac{W_m}{T_m} \frac{T_m}{m W_m} \frac{W_m}{{n \choose m}}= \frac{W_m}{m {n \choose m}}
\]
which finally gives
\[
\mathbb{P}[\pi_{\tau_n} = \mathds{1}] = \sum_{m = \lceil (n+k)/(k+1) \rceil}^{n} \frac{W_m}{m{n \choose m}}.
\]
\end{proof}

For $k=2$ (\ref{eq_opt_probabil}) takes much simpler form:
\[
\mathbb{P}[\pi_{\tau_n} = \mathds{1}] = \sum_{m = \lceil (n+2)/3 \rceil}^{n} \frac{1}{m{n \choose m}} \sum_{h=0}^{\lfloor (n-m)/k\rfloor} {m-1 \choose n-m-h} {n-m-h \choose h}.
\]

Now we will show that $\mathbb{P}[\pi_{\tau_n} = \mathds{1}]$ is of the order $n^{-1/(k+1)}$, $1 \leq k < n$. We will not use (\ref{eq_opt_probabil}). In order to prove this result we use the continuous time approach to arrivals of vertices.

Recall that $V_n = \{v_1, v_2, \ldots, v_{n}\}$ and $E_n = \{(v_i, v_{i-1}), i=2,3,\ldots,n\}$ are the sets of vertices and edges of $P_n$ respectively; thus $v_{1} = \mathds{1}$ is the sink. Note that if $\mathds{1} = \pi_t$ and $v_n$ is still to appear at the time $t$, then at the time $t$ the condition $n-t = k (c(P_{(t)}) - 1) + b_t$ is not satisfied (we have then $n-t > k (c(P_{(t)}) - 1) + b_t$). Thus in order to have $\pi_{\tau_n}=\mathds{1}$ $v_n$ must precede $\mathds{1}$ in $\pi$. Note also that in the two easy cases, when $k=n-2$ or $k=n-1$ this condition is also sufficient and then $\mathbb{P}[\pi_{\tau_n}=\mathds{1}] = 1/2$. Throughout the rest of this section we assume that $1 \leq k < n-2$ (although the case $k=1$ was solved in \cite{GKMM}).

Let us associate with each $v_i$, $i=1,2,\ldots,n$, a random variable $A_i$ of a value drawn uniformly from the interval $[0,1]$, where all $A_i$'s are independent. Let us treat $A_i$ as the time of arrival of $v_i$. We have thus generated the uniform random order of arrivals of vertices from $P_n^k$. The arrival time of the sink will be denoted by $p$ ($A_1=p$). Note that all the permutations of vertices are still equiprobable. Since all $A_i$'s are independent, if the arrival time of the sink is $A_1=p$, the probability that a particular vertex appears before the sink is equal to $p$.

Let us define the following sequence of the indicator random variables
\begin{equation}
\label{def_Xi}
X_i^{(p)} = \left\{  \begin{array}{ll}
1 & \textrm{if} \hspace{3 pt} A_{i+1}>p \wedge A_{i+2}>p \wedge \ldots \wedge A_{i+k+1}>p, \\
0 & \textrm{otherwise,}
 \end{array}\right.
\end{equation}
for $1 \leq i \leq n-k-2$. Let also $X^{(p)} = \sum_{i=1}^{n-k-2}X_i^{(p)}$. The equality $X^{(p)}=0$ means that in the induced graph at the time $p$ there are no two components such that they are neighbours (no other element from the induced graph is between them) and the distance between them in $P_n$ is greater than $k+2$ (by the distance between two components we understand the length of the shortest path in $P_n$ that joins vertices from the different components). Hence $X^{(p)}=0$ and $A_n<p$ ensure that at the time $p$ when the sink comes (suppose $\mathds{1} = \pi_t$) the condition $n - t = k (c(P_{(t)}) - 1) + b_t$ is satisfied. Also if $A_n>p$ or $X^{(p)}>0$ we have $n-t>k (c(P_{(t)}) - 1) + b_t$. Thus $\pi_{\tau_n} = \mathds{1}$ if and only if $X^{(p)}=0$ and $A_n<p$. Thus, since $\mathbb{P}[\pi_{\tau_n} = \mathds{1} | A_1=p] = \mathbb{P}[X^{(p)}=0,A_n<p | A_1=p]$ and all $p$'s are equiprobable, the following lemma holds (compare \cite{BBMS}, Lemma 3.2).

\begin{lemma}
\label{l_integral}
For $P_n^k$ \rm{(}$1 \leq k < n-2$\rm{)} we have
\[
\mathbb{P}[\pi_{\tau_n} = \mathds{1}] = \int\limits_{0}^{1} {\mathbb{P}[X^{(p)}=0, A_n<p|A_1=p] \,dp}.
\]
\end{lemma}
{\begin{flushright} \vspace*{-1mm} \mbox{$\Box$} \end{flushright}}
Recall the definitions of the gamma and beta functions.
\[
\begin{split}
\Gamma(x)=\int_{0}^{\infty} {t^{x-1} \exp\{-t\} \,dt}, \hspace{10pt} B(a,b)=\int_{0}^{1} {t^{a-1}(1-t)^{b-1} \,dt}, \hspace{5pt} x>0, a>0, b>0.
\end{split}
\]

Since $\tau_n$ is optimal any stopping time gives a lower bound for its effectiveness. Let $\tau_p^*$ be defined as follows:

\vspace{0.2cm}

\textit{Flip an asymmetric coin, having some probability $p$ of coming down tails, $n$ times. If it comes down tails $M$ times reject the first $M$ elements. After this time pick the first element which is maximal in the induced graph. In other words, $\tau_p^*$ is equal to the first $j>M$ such that $\pi_j \in Max(P_{(j)})$. If no such $j$ is found let $\tau_p^* = n$.}

\vspace{0.2cm}

The randomization used in the above definition was introduced by Preater in \cite{JP} who used the fact stated in Lemma \ref{l_indep} (see also \cite{MS}, Lemma 3.1).

\begin{lemma}
\label{l_indep}
Let $\pi \in S_n$ be a random permutation of vertices in $V$. Suppose that we have a coin that comes down tails with probability $p$. Let $M$ denote the number of tails in $n$ tosses. Then all vertices from $V$ appear in $\{\pi_1, \pi_2, \ldots, \pi_M\}$ with probability $p$ independently.
\end{lemma}
{\begin{flushright} \vspace*{-1mm} \mbox{$\Box$} \end{flushright}}

\begin{lemma}
\label{th_lower}
Let $P_n^k$ be the $k$th power of a directed path, $1 \leq k=k(n) < n-2$. Let $\pi$ be a random permutation of its vertices and $p = 1-(1-\varepsilon)n^{-1/(k+1)}$ for an $\varepsilon \in (0,1)$. There exists a constant $\tilde{c}>0$ such that
$$
\liminf_{n \rightarrow \infty} n^{1/(k+1)} \mathbb{P}[\pi_{\tau_p^*} = \mathds{1}] \geq \tilde{c}.
$$
\end{lemma}
\begin{proof}
Let $V_p^*$ be the set $\{\pi_1, \pi_2, \ldots, \pi_M\}$ from Lemma \ref{l_indep}.
Let us define the following sequence of the indicator random variables
\[
X_i^{[M]} = \left\{  \begin{array}{ll}
1 & \textrm{if} \hspace{3 pt} \{v_{i+1}, v_{i+2}, \ldots, v_{i+k+1}\} \subseteq V_n \setminus V_p^*, \\
0 & \textrm{otherwise,}
 \end{array}\right.
\]
for $1 \leq i \leq n-k-2$. Let $X^{[M]} = \sum_{i=1}^{n-k-2}X_i^{[M]}$. Note that if $X^{[M]}=0$, $v_n \in V_p^*$ and $\mathds{1} \in V_n \setminus V_p^*$ then $\mathds{1}$ is the only element which comes as the maximal one in the induced graph after time $M$. Thus we obtain
\[
\mathbb{P}[\pi_{\tau_p^*} = \mathds{1}] \geq \mathbb{P} [X^{[M]}=0, v_n \in V_p^*, \mathds{1} \in V_n \setminus V_p^*].
\]
We have $\mathbb{P}[X_i^{[M]}=1]=(1-p)^{k+1}$, therefore $\mathbb{E}X^{[M]} = (n-k-2)(1-p)^{k+1} \leq n(1-p)^{k+1}$. By Markov's inequality $\mathbb{P}[X^{[M]} \geq 1] \leq \mathbb{E}X^{[M]}$, whence
\[
\mathbb{P}[X^{[M]}=0]  = 1 - \mathbb{P}[X^{[M]} \geq 1] \geq 1 - \mathbb{E}X^{[M]} \geq 1 - n(1-p)^{k+1}.
\]
Since $p = 1-(1-\varepsilon)n^{-1/(k+1)}$, by Lemma \ref{l_indep} we obtain
\[
\begin{split}
\mathbb{P}[&\pi_{\tau_p^*} =  \mathds{1}] \geq \mathbb{P}[X^{[M]}=0] \mathbb{P}[v_n \in V_p^*] \mathbb{P}[\mathds{1} \in V_n\setminus V_p^*] \geq \\
& (1-n(1-p)^{k+1}) p (1-p) = (1-(1-\varepsilon)^{k+1})(1-(1-\varepsilon)n^{-1/(k+1)})(1-\varepsilon)n^{-1/(k+1)},
\end{split}
\]
whence
\[
n^{1/(k+1)} \mathbb{P}[\pi_{\tau_p^*} = \mathds{1}] \geq (1-(1-\varepsilon)^{k+1})(1-(1-\varepsilon)n^{-1/(k+1)})(1-\varepsilon),
\]
thus for $1 \leq k < n-2$ we obtain
\[
\liminf_{n \rightarrow \infty} n^{1/(k+1)} \mathbb{P}[\pi_{\tau_p^*} = \mathds{1}] \geq \liminf_{n \rightarrow \infty} (1-(1-\varepsilon)^{k+1})(1-(1-\varepsilon)n^{-1/(k+1)})(1-\varepsilon) = \tilde{c}>0.
\]
\end{proof}

\begin{theorem}
\label{th_opti_prob}
Let $1 \leq k=k(n) < n$. There exists a constant $c>0$ such that for $P_n^k$
$$
\liminf_{n \rightarrow \infty} n^{1/(k+1)} \mathbb{P}[\pi_{\tau_n} = \mathds{1}] \geq c.
$$
\end{theorem}
\begin{proof}
We have already discussed the cases $k=n-2$ and $k=n-1$. Then $\mathbb{P}[\pi_{\tau_n}=\mathds{1}] = 1/2$ and $\lim_{n \rightarrow \infty} n^{1/(k+1)} \mathbb{P}[\pi_{\tau_n} = \mathds{1}] = 1/2$. For the constant $\tilde{c}$ from Lemma \ref{th_lower} and for $1 \leq k <n-2$ by the optimality of $\tau_n$ we have
\[
\liminf_{n \rightarrow \infty} n^{1/(k+1)} \mathbb{P}[\pi_{\tau_n} = \mathds{1}] \geq \liminf_{n \rightarrow \infty} n^{1/(k+1)} \mathbb{P}[\pi_{\tau_p^*} = \mathds{1}] \geq \tilde{c}.
\]
Then setting $c=\min\{\tilde{c}, 1/2\}$ we obtain for $1 \leq k <n$ 
\[
\liminf_{n \rightarrow \infty} n^{1/(k+1)} \mathbb{P}[\pi_{\tau_n} = \mathds{1}] \geq c.
\]
\end{proof}
\begin{theorem}
\label{th_upper}
Let $1 \leq k=k(n) < n$. There exists a constant $C>0$ such that for $P_n^k$
$$
\limsup_{n \rightarrow \infty} n^{1/(k+1)} \mathbb{P}[\pi_{\tau_n} = \mathds{1}] \leq C.
$$
\end{theorem}
\begin{proof}
Recall again that we have already discussed the two easy cases for $k=n-2$ and $k=n-1$ where $\mathbb{P}[\pi_{\tau_n}=\mathds{1}] = 1/2$, whence, obviously,  $\lim_{n \rightarrow \infty} n^{1/(k+1)} \mathbb{P}[\pi_{\tau_n} = \mathds{1}] = 1/2$. Further let $1 \leq k=k(n) <n-2$.

Since the events $[X^{(p)}=0]$ and $[A_1=p]$ are independent and also $[X^{(p)}=0]$ and $[A_1=p, A_n<p]$ are independent, $\mathbb{P}[A_n<p|A_1=p]=p$ and $\mathbb{P}[X^{(p)}=0|A_1=p] = 1-\mathbb{P}[X^{(p)}\geq 1|A_1=p] $, by Lemma \ref{l_integral} we get
\begin{equation}
\label{eq_success}
\begin{split}
\mathbb{P}[\pi_{\tau_n} = \mathds{1}] =& \int\limits_{0}^{1} {\mathbb{P}[A_n<p, X^{(p)}=0|A_1=p] \,dp} =
 \int\limits_{0}^{1} {p(1-\mathbb{P}[X^{(p)}\geq 1|A_1=p]) \,dp} = \\
& \int\limits_{0}^{1} {p \,dp} - \int\limits_{0}^{1} {p\mathbb{P}[X^{(p)}\geq 1|A_1=p] \,dp} =
  1/2 - \int\limits_{0}^{1} {p\mathbb{P}[X^{(p)}\geq 1|A_1=p] \,dp}.
\end{split}
\end{equation}
Now we are going to bound $\mathbb{P}[X^{(p)} \geq 1|A_1=p]$ from below. Let $m = \lfloor \frac{n-2}{k+1}\rfloor$. Since
\[
X_1^{(p)}, X_{(k+1)+1}^{(p)}, X_{2(k+1)+1}^{(p)}, \ldots, X_{(m-1)(k+1)+1}^{(p)}
\]
are independent and $\mathbb{P}[X_i=1|A_1=p]=(1-p)^{k+1}$ for $i=1,2,\ldots,n-k-2$, we have
\begin{equation}
\label{eq_X1}
\begin{split}
\mathbb{P}[X^{(p)} \geq 1|&A_1=p] \geq \\
& \mathbb{P}[X_1^{(p)}=1 \vee X_{(k+1)+1}^{(p)}=1 \vee \ldots \vee X_{(m-1)(k+1)+1}^{(p)}=1|A_1=p] = \\
& 1- \mathbb{P}[X_1^{(p)}=0 \wedge X_{(k+1)+1}^{(p)}=0 \wedge \ldots \wedge X_{(m-1)(k+1)+1}^{(p)}=0|A_1=p] = \\
& 1-(1-(1-p)^{k+1})^{m}.
\end{split}
\end{equation}
Thus from \ref{eq_success} and \ref{eq_X1} we obtain
\[
\begin{split}
\mathbb{P}[\pi_{\tau_n} = \mathds{1}] \leq & 1/2 - \int\limits_{0}^{1} {p(1-(1-(1-p)^{k+1})^{m}) \,dp} = \\
& 1/2 - \int\limits_{0}^{1} {p \,dp} + \int\limits_{0}^{1} {p(1-(1-p)^{k+1})^{m} \,dp} \leq \int\limits_{0}^{1} {(1-(1-p)^{k+1})^{m} \,dp}.
\end{split}
\]
Let us substitute $x = (1-p)^{k+1}$ in the above integral. We get $(1-p)^k = x^{\frac{k}{k+1}}$ and $\,dp = - \frac{1}{k+1} x^{-\frac{k}{k+1}} \, dx$. Therefore
\[
\int\limits_{0}^{1} {(1-(1-p)^{k+1})^{m} \,dp} =  \int\limits_{0}^{1} {\frac{1}{k+1} x^{-\frac{k}{k+1}}(1-x)^m \,dx}
\]
and integrating by parts we get
\[
\begin{split}
\int\limits_{0}^{1} {\frac{1}{k+1} x^{-\frac{k}{k+1}}(1-x)^m\,dx} =& \left[ x^{\frac{1}{k+1}} (1-x)^m \right]^1_0 - \int\limits_{0}^{1} {-x^{\frac{1}{k+1}}m(1-x)^{m-1} \,dx} = \\
& m \int\limits_{0}^{1} {x^{\frac{1}{k+1}}(1-x)^{m-1} \,dx} = m B\left(1+\frac{1}{k+1}, m\right).
\end{split}
\]
Since $m \Gamma(m) = \Gamma(m+1)$ and for every real $a>0$, $b>0$ we have $B(a,b)=\frac{\Gamma(a)\Gamma(b)}{\Gamma(a+b)}$, we obtain
\[
\mathbb{P}[\pi_{\tau_n} = \mathds{1}] \leq m B\left(1+\frac{1}{k+1}, m\right)= \frac{m \Gamma(m) \Gamma(1+\frac{1}{k+1})}{\Gamma(m+1+\frac{1}{k+1})} = \frac{\Gamma(m+1) \Gamma(1+\frac{1}{k+1})}{\Gamma(m+1+\frac{1}{k+1})}.
\]
Thus we have
\[
\limsup_{n \rightarrow \infty} n^{1/(k+1)} \mathbb{P}[\pi_{\tau_n}= \mathds{1}] \leq \limsup_{n \rightarrow \infty} n^{1/(k+1)} \frac{\Gamma(m+1) \Gamma(1+\frac{1}{k+1})}{\Gamma(m+1+\frac{1}{k+1})}
\]
where $m = \lfloor \frac{n-2}{k+1}\rfloor$. We need to show that there exists a constant $C>0$ such that
\begin{equation}
\label{ineq_gamma}
\limsup\limits_{n \rightarrow \infty} n^{1/(k+1)} \frac{\Gamma(m+1)}{\Gamma(m+1+\frac{1}{k+1})} \leq C.
\end{equation}
Since $\frac{\Gamma(x + \alpha)}{\Gamma(x)}$ is increasing in $x$ we can write
\[
\frac{\Gamma(m+1)}{\Gamma(m)} = \frac{\Gamma(m+1)}{\Gamma(m+1-\frac{1}{k+1})} \frac{\Gamma(m+1-\frac{1}{k+1})}{\Gamma(m+1-\frac{2}{k+1})} \ldots \frac{\Gamma(m+1-\frac{k}{k+1})}{\Gamma(m)} \leq \left(\frac{\Gamma(m+1+\frac{1}{k+1})}{\Gamma(m+1)}\right)^{k+1}.
\]
Thus using $\Gamma(m+1)/\Gamma(m)=m$ we obtain
\[
m^{\frac{1}{k+1}} = \left\lfloor \frac{n-2}{k+1}\right\rfloor^{\frac{1}{k+1}} \leq \frac{\Gamma(m+1+\frac{1}{k+1})}{\Gamma(m+1)}.
\]
Thus (\ref{ineq_gamma}) holds when $k$ is a constant and when $k=k(n) \xrightarrow{n \rightarrow \infty} \infty$ and by a standard contradiction type argument it also holds in general.

Note that for $k<n-2$ we actually obtain
\[
\limsup_{n \rightarrow \infty} n^{1/(k+1)} \mathbb{P}[\pi_{\tau_n}= \mathds{1}] \leq \limsup_{n \rightarrow \infty} \Gamma(1+\frac{1}{k+1}) (k+1)^{\frac{1}{k+1}}.
\]
The function $f(k) = \Gamma(1+\frac{1}{k+1}) (k+1)^{1/(k+1)}$ ($k \in \mathbb{N}$) attains its maximum at $k=2$ thus we can set the constant $C = \Gamma(4/3)3^{1/3} \approx 1.29$.
\end{proof}

Although we do not know the optimal algorithm when the selector does not know the values of $d_P$ in the induced graph, we know the order of the probability of its success.
\begin{corollary}
For $P_n^k$ being the $k$th power of a directed path let $\tilde{\tau}_n$ be the optimal stopping time for choosing the sink when the selector does not know the values $d_P$ of each edge that appears in the induced graph. Then
$$
\mathbb{P}[\pi_{\tilde{\tau}_n} = \mathds{1}] = \Theta(n^{-1/(k+1)}).
$$
\end{corollary}
\begin{proof}
We have $\mathbb{P}[\pi_{\tilde{\tau}_n} = \mathds{1}] \leq \mathbb{P}[\pi_{\tau_n} = \mathds{1}]$ because when the values of $d_P$ are known in the induced graph one can take at least as efficient decision as when they are not known. On the other hand note that our lower estimation of $\liminf_{n \rightarrow \infty} n^{1/(k+1)} \mathbb{P}[\pi_{\tau_n} = \mathds{1}]$ in Theorem \ref{th_opti_prob} does not use the information about the values $d_P$ at all. Thus the estimation is also true for $\tilde{\tau}_n$.
\end{proof}
\begin{remark}
{\rm Recall that for $k=1$ our problem is the directed path case from \cite{GKMM}. For the other extreme case $k=n-1$ it is the classical linear order secretary problem with extra information $d_{P^{n-1}_{n}}(e)$. From \cite{DVL} we know that the probability of success of the optimal algorithm for the linear order is asymptotically $1/e$. It is quite surprising that revealing this additional information about distances increases the probability of success of the optimal algorithm only to $1/2$.}
\end{remark}

\end{document}